\documentclass[11pt]{amsart}

\usepackage{a4wide,amsmath,amssymb}
\usepackage[toc,page]{appendix}
\usepackage{cancel}
\usepackage{url}
\usepackage{hyperref}
\usepackage{xcolor}

\makeatletter
\@namedef{subjclassname@2020}{\textup{2020} Mathematics Subject Classification}
\makeatother

\newtheorem{thm}{Theorem}[section]
\newtheorem{lemma}[thm]{Lemma}

\newtheorem{cor}[thm]{Corollary}
\theoremstyle{remark}
\newtheorem{rem}[thm]{Remark}

\theoremstyle{definition}

\newtheoremstyle{Claim}{}{}{\itshape}{}{\itshape\bfseries}{:}{ }{#1}
\theoremstyle{Claim}

\newcommand{\R}{\mathbb{R}}

\newcommand{\aac}{\`a}

\newcommand{\eps}{\varepsilon}

\theoremstyle{plain}

 \numberwithin{equation}{section}

\begin{document}

\title[]{Gradient estimates for quasilinear elliptic Neumann problems with unbounded first-order terms}
\author{Marco Cirant}
\address{Dipartimento di Matematica ``Tullio Levi-Civita'', Universit\`a degli Studi di Padova, 
via Trieste 63, 35121 Padova (Italy)}
\curraddr{}
\email{cirant@math.unipd.it}
\thanks{}
\author{Alessandro Goffi}
\address{Dipartimento di Matematica e Informatica ``Ulisse Dini'', Universit\`a degli Studi di Firenze, 
viale G. Morgagni 67/a, 50134 Firenze (Italy)}
\curraddr{}
\email{alessandro.goffi@unifi.it}
\thanks{}
\author{Tommaso Leonori}
\address{Dipartimento di Scienze di Base ed Applicate per l'Ingegneria, Sapienza Universit{\aac} di Roma, Via Antonio Scarpa 10, 00161 Roma (Italy)}
\curraddr{}
\email{tommaso.leonori@uniroma1.it}
\thanks{}
\thanks{The authors are members of the Gruppo Nazionale per l'Analisi Matematica, la Probabilit\`a e le loro Applicazioni (GNAMPA) of the Istituto Nazionale di Alta Matematica (INdAM). M. Cirant and A. Goffi were partially supported by the INdAM-GNAMPA Project 2022 ``Propriet\`a quantitative e qualitative per EDP non lineari con termini di gradiente'', the Project funded by the EuropeanUnion – NextGenerationEU under the National Recovery and Resilience Plan (NRRP), Mission 4 Component 2 Investment 1.1 - Call PRIN 2022 No. 104 of February 2, 2022 of Italian Ministry of University and Research; Project 2022W58BJ5 (subject area: PE - Physical Sciences and Engineering) ``PDEs and optimal control methods in mean field games, population dynamics and multi-agent models" and by the King Abdullah University of Science and Technology (KAUST) project CRG2021-4674 ``Mean-Field Games: models, theory and computational aspects". A. Goffi and T. Leonori were partially supported by the INdAM-GNAMPA Project 2023 ``Problemi variazionali/nonvariazionali: interazione tra metodi integrali e principi del massimo''. This research was partially carried out while A. Goffi was Postdoctoral  research fellow at ``Dipartimento di Scienze di Base ed Applicate per l'Ingegneria, Sapienza Universit{\aac} di Roma''.}

\date{\today}

\begin{abstract}
This paper studies global a priori gradient estimates for divergence-type equations patterned over the $p$-Laplacian with first-order terms having power-growth with respect to the gradient under suitable integrability assumptions on the source term of the equation. The results apply to elliptic problems with unbounded data in Lebesgue spaces complemented with Neumann boundary conditions posed on convex domains of the Euclidean space. 
\end{abstract}

\maketitle

\section{Introduction}
A well-known result in the theory of linear elliptic equations states that any strong solution to the Poisson equation $-\Delta u=f\in L^q$ posed on a  bounded open set $\Omega$ of $\R^N$ with enough regular boundary satisfies the so-called maximal $L^q$-regularity estimate, i.e. an estimate on $\|D^2u\|_{L^q }$ holds in terms of $\|f\|_{L^q }$, with linear dependence. Then, optimal gradient estimates follow by Sobolev embeddings, depending on the range of the exponent $q$ with respect to the dimension $N$ of the ambient space.\\
The aim of this manuscript is to provide a quasilinear counterpart of these maximal $L^q$-regularity properties for a class of quasilinear elliptic boundary-value problems with diffusion in divergence form and lower-order terms with power growth in the gradient. The class of diffusions we are able to encompass is patterned over the $p$-Laplacian, the main model being
\[
\lambda u-\mathrm{div}(|Du|^{p-2}Du)+|Du|^\gamma=f(x)\quad \text{ in }\Omega\ ,
\]
for $p\in(1,\infty)$, $\gamma>p-1$,  $f\in L^q(\Omega)$ for some $q>1$ and $\lambda\geq0$. For suitable solutions $u:\Omega\to\R$ (obtained by approximation) to Neumann boundary-value problems posed on convex $C^2$ domains, our main results give, for any $\gamma>p-1$, the following a priori estimates
\begin{equation}\label{quasilip}
p>1\,, \quad f\in L^N(\Omega)  \implies Du\in L^r(\Omega),\quad  1\leq r<\infty, 
\end{equation}
and 
\begin{equation}\label{maximal}
p\geq 2\,, \quad f\in L^q(\Omega)\ ,\quad q\geq \frac{N(\gamma-(p-1))}{\gamma}\text{ and } \ q>2  \implies |Du|^\gamma\in L^q(\Omega).
\end{equation}
In the case $p=2$, estimate \eqref{quasilip} was proved in \cite{Lions85}, whilst an estimate of the form \eqref{maximal} was conjectured by P.-L. Lions \cite{Napoli,LionsSeminar} and has been the focus of a recent intensive research, especially in connection with the analysis of Mean Field Games systems: the recent work \cite{CGell} addressed the conjecture of maximal regularity for the viscous problem in the periodic case, the later developments in \cite{GP,G} treated global regularity for boundary-value problems with Neumann and Dirichlet boundary conditions respectively, while interior estimates in the superquadratic regime were the subject of the paper \cite{CV} through a different approach based on a blow-up argument. Along this line, we mention the analysis of the parabolic problem carried out in \cite{CGpar} by means of a rather different (nonlinear) duality method, even combined with the Bernstein technique \cite{cg20}. The recent paper \cite{c22par} addresses time-dependent problems with superquadratic nonlinearity via blow-up and duality methods. Finally, the work \cite{Gnondiv} contains interior estimates for stationary and parabolic equations with quadratic growth and diffusion in nondivergence form.\\
A peculiar feature of our results is that they hold for a degenerate/singular diffusion in the so-called supernatural growth regime of the first-order term, i.e. when $\gamma>p$. Though the sublinear and the subnatural growth, respectively $\gamma<p-1$ and $\gamma<p$, have been widely analyzed for many years, see for example \cite{Serrin1,Serrin2,Simon,Wang,GT,Lieb91,LUbook,AmannCrandall,Maugeri} and the references therein, the literature is to our knowledge poor in the supernatural growth regime. In particular, few results are available for PDEs with first-order terms having power-like growth in the gradient: Lipschitz bounds were obtained for fully nonlinear singular equations in \cite{BDL} when $f\in W^{1,\infty}$ through viscosity solutions' methods and by pointwise Bernstein arguments in \cite{LPcpde}, while H\"older estimates for distributional semi-solutions were obtained for $L^q$ or more general Morrey right-hand sides in \cite{Glio}. These results have their roots in the earlier research carried out for the case $p=2$ in \cite{Lions80,CDLP,DP}. Nonetheless, we mention that some of the results we obtain, especially concerning \eqref{maximal}, are new even in the regime $p-1<\gamma<p$ and/or $q\leq N$. 

An additional distinctive feature with respect to the previous works on the subject is that our arguments cover all the range of exponents $q$ up to the endpoint threshold $q=\frac{N(\gamma-(p-1))}{\gamma}$, for any $\gamma>p-1$. In such a limiting case, we also enlighten the role of zero-th order term.  This was first observed in \cite{CGpar} for parabolic problems and then in \cite{G} for elliptic equations equipped with Dirichlet boundary conditions in the case $p=2$, when the first-order term has subquadratic growth. The parabolic superquadratic case with linear diffusion has been recently addressed in \cite{c22par}, up to the endpoint threshold. We emphasize that such a lower bound on the summability exponent $q$ is in general necessary for the validity of the maximal regularity property, see \cite{CGell} for a counterexample in the linear case $p=2$, and even for existence issues \cite{HMV}.\\
Finally, we emphasize that in contrast to compactness-based methods \cite{GigaTsubouchi,CV}, our approach is, in some cases, able to provide quantitative bounds, see e.g. the estimate in the next Theorem \ref{main1}. This is related with the validity of the so-called ``strong'' maximal regularity for such nonlinear PDEs, which is at this stage widely open unless $p=2$ and $n=2$, cf. \cite{Napoli}. These kind of stronger higher regularity properties were studied in Theorem 4.3 of \cite{CianchiMazyaJEMS} for $-\Delta_pu=f$.\\

The approach used in our main results (Theorem \ref{main1} and Theorem \ref{main2}) is based on the so-called integral Bernstein method, see \cite{Lions85} and the later papers \cite{LasryLions,BardiPerthame} together with the recent developments \cite{CGell,GP,cg20}. 

In particular, the basic idea in Theorem \ref{main1} relies on using a $\hat p$-Laplacian of $u$ as a test function in the weak formulation of the problem, with a suitably large $\hat p$. Theorem \ref{main2} exploits a delicate argument that still revolves around testing by a similar function, but also exploits a continuity argument that hinges on the integration on super-level sets of the gradient. This latter technique has been inspired by \cite{CGell}, and the previous work on integral estimates for solutions to quasilinear elliptic problems in the subnatural regime \cite{GMP14}. Unfortunately, to apply such a technique we need to add the hypothesis  $p\geq2$. 
Both the results exploit the coercivity of the gradient term via a  weighted Bochner identity. Roughly speaking, for the $p$-Laplacian diffusion in nondivergence form
\[
|Du|^{p-2}\left(\Delta u+(p-2)\frac{\Delta_\infty u}{|Du|^2}\right)=:|Du|^{p-2}\mathcal{A}(D^2u)
\] 
we have the following identity solved by $w=|Du|^2$ (considering the leading operator $|Du|^{p-2}\Delta u$ and $|Du|^{p-2}$ as a ``coefficient'')
\[
|Du|^{p-2}\Delta w=2|Du|^{p-2}|D^2u|^2+2|Du|^{p-2}Du\cdot D\Delta u.
\]
From this, we use a ``generalized'' Cauchy-Schwarz inequality
\[
|D^2u|\geq c(N,p)\mathcal{A}(D^2u),
\]
plug the equation and exploit the lower bound
\[
|Du|^{p-2}|D^2u|^2\geq \left(c_1\frac{|Du|^{2\gamma}}{2}-c_2(f-\lambda u)^2\right)|Du|^{2-p}.
\]
The last inequality shows that (part) of the second order term grants additional coercivity, crucial to conclude the higher-regularity properties for this class of equations. This term is also fundamental to deduce a second order estimate and, notably, obtain in the limit $\gamma=0$ some known properties for the $p$-Poisson equation via a different approach, cf. Remark \ref{so}. We believe this is a neat difference with respect to classical references dealing with the integral Bernstein method for quasilinear equations, cf. \cite{CianchiMazyaCPDE,DiB,DiBFriedman,GL}. The derivation of this chain of inequalities will be discussed in detail in Lemma \ref{diff2}.

\smallskip 

Since solutions to our problems, even without gradient dependent terms, are in general no more regular than $C^{1,\alpha}(\overline{\Omega})$ for some $\alpha\in(0,1)$, we use an approximation procedure considering the uniformly elliptic problem
\[
\lambda u_\eps-\mathrm{div}((\eps+|Du_\eps|^2)^{\frac{p-2}{2}}Du_\eps)+H(Du_\eps)=f_\eps(x) \qquad \mbox{ in } \ \Omega
\]
where $\eps>0$,  $\lambda\geq0$ and $f_\eps$ is a smooth approximation of $f$, which admits a smooth solution $u_\eps$ (see \cite{LPcpde}), and prove estimates independent of $\eps$.\\

\medskip 

As far as gradient estimates for quasilinear elliptic equations are concerned, this classical problem has been extensively analyzed, especially when the equation is driven by the sole $p$-Laplacian. Classical works on the subject typically consider problems up to the natural growth $\gamma = p$, and include for example the papers \cite{U,BocGal,DibNA,DM,Iwaniec,Lewis,Evans,Lieb91}, that also study higher regularity properties at the level of $C^{1,\alpha}$ spaces. More recent papers \cite{MingioJEMS,CianchiMazyaCPDE,CianchiMazyaJEMS,CianchiMazyaARMA,BMCCM,BeckMingione,Brasco} focused on gradient regularity estimates with unbounded right-hand side at the level of Lebesgue and Lorentz classes, treating both global and local bounds. Some recent works have been also concerned with the optimal second-order regularity for such problems, see \cite{BrascoSantambrogio,CianchiMazyaJMPA,CianchiMazyaARMA2,BeiraoCrispo,Dong}, or even to study the gradient regularity of solutions driven by the mixed operator $-\Delta_1-\Delta_p$, see \cite{GigaTsubouchi,Tsubouchi}.\\
Regarding the assumptions on the integrability of $f$ in the model case of the $p$-Poisson equation (possibly perturbed with gradient terms having sublinear growth), interior  Lipschitz bounds have been studied in \cite[Theorem 1 and the subsequent Remark]{DibNA} and \cite[Remark 7.4]{DiBFriedman} under the assumption that $q>\frac{Np}{p-1}$. Interior estimates in $W^{1,\infty}$ have been then obtained in \cite{Lieb91} under the weaker integrability assumption $q>N$, while the work \cite{Lieb93} assumes $f$ controlled in the Morrey class $\mathcal{L}^{1,s}$, $s>N$. Optimal gradient regularity estimates for boundary-value problems of the $p$-Poisson equation can be found in \cite{CianchiMazyaCPDE,CianchiMazyaJEMS}.\\
We further mention that the case of slowly increasing first-order terms, e.g. when $\gamma\leq p-1$, has been already treated via techniques from nonlinear potential theory, see \cite{KMrmi} and the references therein, but when the right-hand side datum $f\in L^\infty$ and $p\geq 2$. In the case $p<2$, further restrictions on the growth $\gamma$ have been imposed, see again \cite{KMrmi}, at least in the parabolic framework. We refer to \cite{DuzaarMingioneLincei,DuzaarMingioneCalcVar,DuzaarMingioneAM,MingioJEMS,KMsurvey,MingioPalaSurvey} for more details on the literature of nonlinear potential theory. Earlier results for parabolic problems driven by the $p$-Laplacian with first-order gradient terms growing at most as $|Du|^{p-1}$ and unbounded source terms in Lebesgue spaces can be found in \cite[Chapter VIII, Section 1-(ii)]{DiB}, under the restriction $f\in L^q_{x,t}$, $q>N+2$.\\
Preliminary forms of the integral Bernstein method appeared in \cite{DiB,DiBFriedman,U}, but we point out that they adapt at most for sublinear powers $\gamma$ of the gradient. Our refinement of the Bernstein technique, instead, allows to handle problems with coercive, in fact supernatural, gradient terms.\\
 More recently, gradient estimates in the framework of renormalized/approximated solutions in terms of right-hand sides in Marcinkiewicz spaces have been studied in \cite{AlvinoAnnali,4Nap}, see also the related work \cite{PhucAdv} and the references therein.\\

We conclude by saying that the convexity assumption on $\Omega$ in Theorems \ref{main1} and \ref{main2} allows us to give a sign on the boundary integrals coming from the diffusion term, since the second fundamental form on the boundary of a convex set is semidefinite, see Lemma \ref{signw}. We believe that this constraint can be removed using some test function argument as in \cite{LPcpde,PorrCCM}, see Remark \ref{conv}. The literature sometimes encompasses domains less regular than $C^2$, cf. e.g. \cite{Grisvard}, but we do not pursue this direction, referring to Remark \ref{regdom} for further references and discussions. Still, we do not know whether the restriction $p\geq2$ in deriving an estimate like \eqref{maximal} is really necessary or a drawback of our method.

\medskip
\textit{Plan of the paper}. Section \ref{s_N} will be devoted to the proof of gradient estimates in the full range $p> 1$, and integrability of $f$ close to $N$, while Section \ref{s_max} will address the full range of integrability of $f$, but under the restriction $p \ge 2$. Note that the second case develops the arguments which are used in the first case, so it might be useful to start with Section \ref{s_N} and then proceed with Section \ref{s_max} in order to get acquainted with the technicalities. Moreover, having the two sections different developments in a few steps, two slightly different sets of assumptions will be used (though they will both include the model problem described above).

\section{Gradient estimates for $p > 1$ and $q$ close to $N$}\label{s_N}
We consider the following problem
\begin{equation}\label{divappHJ}
\begin{cases}
-\mathrm{div}(a(|D u|^2)D u) + H(Du)=f(x)&\text{ in } \Omega,\\
\partial_\nu u=0&\text{ on } \partial\Omega,
\end{cases}
\end{equation}
where $\nu$ denotes the outward unit vector on $\partial\Omega$ and $a:[0,\infty)\to[0,\infty)$ is of class $C^1(0,\infty)$. We also assume that there exist $p>1$ and constants $\bar c_a,\bar C_a,>0$, $C_a \in \R$ such that 
\begin{equation}\tag{A1}\label{Aa}
-1<\inf_{t>0}\frac{2ta'(t)}{a(t)}\leq \sup_{t>0}\frac{2ta'(t)}{a(t)}\leq C_a<\infty\,,
\end{equation}
and 
\begin{equation}\tag{A2}\label{A1}
\bar c_at^{\frac{p-2}{2}}\leq a(t)\leq \bar C_at^{\frac{p-2}{2}}\,,\quad \forall t>0
\end{equation}
Notice that \eqref{Aa} implies (in particular where $a'(t)<0$) the existence of a constant $\widetilde{c}_a>0$ such that
\begin{equation}\tag{A3}\label{A2}
2ta'(t)+a(t)\geq \widetilde{c}_a a(t)\ ,\quad \forall t>0
\end{equation}
and
\begin{equation}\tag{A4}\label{A3}
t\left|\frac{a'(t)}{a(t)}\right|\leq \frac{C_a}{2}\,,\quad \forall t>0. 
\end{equation}

As far as the lower order term is concerned,  we assume that $H=H(\xi)\in C^2(\R^N \setminus \{0\})\cap C^0(\R^N  )$ and $H$ is radial, i.e. $H(\xi)=h(|\xi|)$ for some $h\in C^2(\R^+)\cap C^0[0,+\infty)  $ and that there exists  constants $ c_H, C_H>0$, $\gamma>1$, such that
\begin{equation}\tag{H1}\label{H1}
c_H|\xi|^{\gamma}\leq H(\xi), \quad \forall \xi \in \R^N\,,
\end{equation}
and
\begin{equation}\tag{H2}\label{H2}
 |H_{\xi\xi}(\xi)|\leq C_H|\xi|^{\gamma-2}\,, \qquad \forall \xi \in \R^N \setminus \{0\}.
\end{equation}

Finally we consider   the right-hand side datum  in the following way: 
\begin{equation}\tag{F}\label{F}
f\in L^q(\Omega) \quad \mbox{ for some  } \quad  q \geq N.
\end{equation}
The model of nonlinear equation that we have in mind   is the following: 
\[
-\mathrm{div}(|Du|^{p-2}Du)+|Du|^\gamma=f(x) \qquad \mbox{ in } \Omega.
\]
Observe that in particular, the $p$-Laplacian satisfies the previous hypotheses with  $a(t)=t^{\frac{p-2}{2}}$, so that $\inf_{t>0}\frac{2ta'(t)}{a(t)}=\sup_{t>0}\frac{2ta'(t)}{a(t)}=p-2$, $\widetilde{c}_a=p-1$ and  $\bar c_a=\bar C_a=1$. 

As already mentioned in the introduction, we cannot expect solutions to \eqref{divappHJ} to be  more regular than $C^{1,\alpha}(\overline{\Omega})$, so that we need to argue by approximation. Thus we  consider the approximated problem driven by a uniformly elliptic operator $a=a(t)$ with $t=|Du|^2+\eps$, $\eps>0$, and a regularized right-hand side $f_\eps\in C^\infty(\Omega)$ (for example, a regularization of $f$ by a convolution with a smoothing kernel), namely
\begin{equation}\label{divappHJ2}
\begin{cases}
-\mathrm{div}(a(|D u|^2+\eps)D u) + H(Du)=f_\eps(x)&\text{ in } \Omega,\\
\partial_\nu u=0&\text{ on } \partial\Omega .
\end{cases}
\end{equation}

Our goal is to prove   estimates that are independent from $\eps$, so that we can inherit them in the limit.

Before stating our first result, let us observe that for any $0 \le \eps \le 1$, we have
\[
\frac{c_H}2(\eps+|\xi|^2)^{\frac{\gamma}{2}} - c_\eps \leq H(\xi) \qquad \forall \xi \in \R^N \setminus \{0\}\,, 
\]

for some $c_\eps$ which vanishes as $\eps \to 0$; hence without loss of generalization we can  subtract such a constant in both sides of \eqref{divappHJ2} and turn \eqref{H1} into  
 \begin{equation}\label{Hbelbis}
\frac{c_H}2(\eps+|\xi|^2)^{\frac{\gamma}{2}} \leq H(\xi) \,, \qquad \forall \xi \in \R^N \setminus \{0\}\,.
\end{equation}
We now state the main result of this section. 
\begin{thm}\label{main1}
Let $\Omega\subset \R^N$, $N\geq 3$, be a $C^2$ convex domain and assume that \eqref{Aa}, \eqref{A1}, \eqref{H1}, \eqref{H2}, \eqref{F}   hold true with $\gamma>p-1$. Then, for any $\eta>1$ large enough,  
there exist two  positive constants $C_{\eta,1}, C_{\eta,2}$ depending on $p,N,c_H, C_H,\bar c_a,\bar C_a,\tilde c_a,C_a,|\Omega|,\eta$, but not on $\eps$, such that 
any smooth solution $u_\eps$ to \eqref{divappHJ2} satisfies  
\begin{equation}\label{est1}
\|Du_\eps\|_{L^\eta(\Omega)}\leq C_{\eta,1}+C_{\eta,2}\|f_\eps\|_{L^{q_{\eta}}(\Omega)}^\frac{1}{p-1}, \qquad \mbox{ with } \quad q_\eta \nearrow N   \qquad \mbox{ as } \quad \eta\to+\infty\,.
\end{equation}
\end{thm}

\medskip 

From now on we drop the subscript $\eps$ for brevity. 

\medskip 

\begin{rem}
Due to the inclusion among Lebesgue spaces, as a byproduct of estimate \eqref{est1} we get \eqref{quasilip}.
\end{rem}

\begin{rem}
We remark that Theorem \ref{main1} leads to a (nonlinear) bound with sublinear dependence on the right-hand side datum. This is in line with the results for the $p$-Poisson equation found in \cite[Theorem 4.3-(ii)]{CianchiMazyaJEMS}. \end{rem}

We premise the following standard result that will allow to handle the boundary integrals, referring for the proof to \cite[Lemma 2.3]{PorrCCM}.
\begin{lemma}\label{signw}
Let $u \in C^2 (\overline{\Omega})$ be such that $\partial_\nu u=0$ on $\partial\Omega$. If $\Omega$ is convex, then $\partial_\nu |Du|^2\leq0$ on $\partial\Omega$.
\end{lemma} 
The proof of the a priori estimate will be accomplished through some preliminary lemmas. In the sequel we will assume all the hypotheses listed in Theorem \ref{main1}. 

\medskip 

Let us denote by 
 \[w=|Du|^2+\eps\] 
 and let us consider the equation in \eqref{divappHJ2}. Multiplying it  by a smooth enough function $\varphi$ and integrating by parts, we deduce, since $u$ satisfies the Neumann boundary condition,  that $u$ solves  the following identity: 
\begin{equation}\label{vareq}
\ \int_\Omega a(w) D u \cdot D \varphi=\int_\Omega \big(f(x)-H(Du(x))\big)\varphi\,.
\end{equation}

The main idea in order to prove Theorem \ref{main1} is to choose 
 $$\varphi=-2\mathrm{div}\big(Du (\eps+|Du|^2)^\beta\big)=-2\sum_{j=1}^N\partial_{x_j} \big(\partial_{x_j}u\ w^\beta\,\big)\,,$$  with $\beta>1$ to be specified later: this essentially amounts to test against a multiple of a regularized $(2\beta)$-Laplacian. 

\medskip 

We first deal with the diffusive term and obtain the following bound from below.

\begin{lemma}\label{diff1}
There exist constants $\zeta_1,\zeta_2>0$,   depending on $\widetilde{c}_a, \bar c_a$, such that 
\begin{equation}\label{l25}
\int_\Omega a(w)Du\cdot D\varphi\,dx\geq \zeta_1\int_\Omega a(w)|D^2u|^2w^{\beta}\,dx+\beta \zeta_2\int_\Omega |Dw|^2w^{\beta-1+\frac{p-2}{2}}\,dx.
\end{equation}
\end{lemma}
\begin{proof}
 We use that $\partial_\nu |Du|^2\leq0$ on $\partial \Omega$ by Lemma \ref{signw}, and  after integrating by parts to get
\begin{align*}
 \sum_{i=1}^N\int_\Omega a(w)\partial_{x_i}u\partial_{x_i}\varphi\,dx &=-2\sum_{i,j=1}^N\int_\Omega a(w)\partial_{x_i}u\partial_{x_i}\big(\partial_{x_j}(\partial_{x_j}u\ w^\beta)\big)\,dx\\
&\geq 2\sum_{i,j=1}^N\int_\Omega\partial_{x_j}\big(a(w)\partial_{x_i}u\big)\partial_{x_i}\big(\partial_{x_j}u\ w^\beta\big)\,dx\\
&=2\sum_{i,j=1}^N\int_\Omega \left[a'(w)\partial_{x_j}w\partial_{x_i}u+a(w)\partial_{x_ix_j}u\right]\cdot\left[\partial_{x_ix_j}u\ w^\beta+\beta\partial_{x_j}u\ \partial_{x_i}w\  w^{\beta-1}\right]\,dx\\
&=  
2\int_\Omega a(w)|D^2u|^2w^{\beta}\,dx
+
2\beta\int_\Omega(Dw\cdot Du)^2a'(w)w^{\beta-1}\,dx
\\
&
+\int_\Omega \beta a(w)   w^{\beta-1} |Dw|^2\,dx
+\int_\Omega  a' (w) w^{\beta} |Dw|^2\,dx  \quad := J
\,.
\end{align*}
 It is easy to observe that when  $a'\geq0$ then  
\[
 {J} \geq 2\int_\Omega a(w)|D^2u|^2w^{\beta}\,dx +  \beta  \int_\Omega |Dw|^2 a(w) w^{\beta-1} \,dx.
\]
On the contrary, when  $a'(t)<0$ we have,   since $|D w|^2\leq 4w |D^2u|^2$ and by \eqref{A2}  
\begin{multline*}
 {J} \geq 
\beta  \int_\Omega \Big[   a(w) +2  w a'(w)        \Big] |Dw|^2  w^{\beta-1} \,dx 
+ 2 \int_\Omega a(w)|D^2u|^2w^{\beta}\,dx
+ 4\int_\Omega a'(w)  w |D^2u|^2w^{\beta}\,dx
\\
 \geq  \beta \widetilde{c}_a   \int_\Omega  a(w)|Dw|^2w^{\beta-1 }\,dx 
 + 
2 \widetilde{c}_a \int_\Omega a(w)|D^2u|^2w^{\beta}\,dx\,.
\end{multline*}
Using now \eqref{A1}, we deduce that there exist  $\zeta_1, \zeta_2 >0$,   depending on $\widetilde{c}_a,\bar c_a$, such that \eqref{l25} holds true. 
\end{proof}

We now elaborate the first term on the right-hand side of the inequality in Lemma \ref{diff1} using the non-variational formulation of the approximated problem.
\begin{lemma}\label{diff2}
There exists $c_1$ depending on $\tilde c_a,C_a,N,p, c_H, C_H$ such that
\begin{equation}\label{l26} 
\zeta_1\int_\Omega a(w)|D^2u|^2w^{\beta}\,dx\geq\frac{\zeta_1}{2}\int_\Omega |D^2u|^2w^{\beta+\frac{p-2}{2}}\,dx+ c_1\int_\Omega\left(\frac{w^\gamma}{4}-2f^2\right)w^{\beta+\frac{2-p}{2}}.
\end{equation} 
\end{lemma}
\begin{proof}
First, we observe that the Cauchy-Schwarz inequality implies
\begin{equation}\label{cz}
\frac{1}{\sqrt{N}}|\Delta u|\leq |D^2u|
 \qquad \text{ and } \qquad 
\left|\frac{\Delta_\infty u}{\eps+|Du|^2}\right|\leq |D^2u|,
\end{equation}
where $\Delta_\infty u=D^2uDuDu$. This allows us to  write the diffusion operator in non-divergence form as
\[\mathrm{div}(a(w)Du)
=a(w)\Delta u+2a'(w)D^2uDu\cdot Du
=a(w)\underbrace{\left[\Delta u+\frac{2 \ a'(w)}{a(w)}\Delta_\infty u\right]}_{:=\mathcal{A}(D^2u)}.
\]
Hence by  \eqref{cz} and     \eqref{A3} we get
\[
|\mathcal{A}(D^2u)|=\left|\Delta u+\frac{2 \ a'(w)}{a(w)}D^2uDuDu\right|\leq \sqrt{N}|D^2u|+ C_a \frac{|D^2u||Du|^2}{\eps+|Du|^2}\leq (\sqrt{N}+ C_a )|D^2u|\,,
\]
and setting $\nu=\frac{1}{\sqrt{N}+C_a }$, the above inequality yields to
\[
|D^2u|\geq \nu|\mathcal{A}(D^2u)|.
\]
We now exploit the non-divergence formulation of the equation, i.e. we rewrite the equation  as 
 \[a(w)\mathcal{A}(D^2u)=H(Du)-f(x)  \quad \mbox{ in } \Omega, \] 
 so that we can plug the equation back in the term $a(w)|D^2u|^2$. 
 Indeed, using the algebraic inequality $(A-B)^2\geq \frac{A^2}{2}-2B^2$, $A,B\in\R$, together with \eqref{Hbelbis} and \eqref{A1}, we deduce that 
 \begin{equation*}
\begin{array}{c} 
\displaystyle a(w)|D^2u|^2\geq \nu^2\frac{[a(w)\mathcal{A}(D^2u)]^2}{a(w)}
\geq \frac{\nu^2}{\bar C_a}\left[H(Du)-f\right]^2w^{\frac{2-p}{2}}
\\ \\
\displaystyle 
\geq \frac{\nu^2}{\bar C_a}\left[\frac{H^2(Du)}{2}-2|f|^2\right]w^{\frac{2-p}{2}}
= \frac{\nu^2c_H^2}{\bar C_a}\left\{\frac{w^\gamma}{8}-2|f|^2\right\}w^{\frac{2-p}{2}}, 
\end{array}
\end{equation*}
so that \eqref{l26} holds true. 
\end{proof}
We now focus on the second term of the right-hand side appearing in  \eqref{l25}.
\begin{lemma}\label{diff3}
There exists $\zeta_3,\zeta_4>0$ depending on $\Omega,\zeta_2,\beta,p$ such that
\[
 \beta\zeta_2\int_\Omega |Dw|^2w^{\beta+\frac{p}{2}-2}\,dx\geq \zeta_3\left(\int_\Omega w^{\left(\beta+\frac{p}{2}\right)\frac{N}{N-2}}\,dx\right)^{\frac{N-2}{N}}-\zeta_4\int_\Omega w^{\beta+\frac{p}{2}}\,dx.
 \]
\end{lemma}
\begin{proof}
We have by the Sobolev inequality
\begin{multline*}
 \beta\zeta_2\int_\Omega |Dw|^2w^{\beta+\frac{p}{2}-2}\,dx=\frac{4\beta\zeta_2}{(\beta+\frac{p}{2})^2}\int_\Omega |Dw^{\frac{\beta+\frac{p}{2}}{2}}|^2\,dx\\
\geq C_{\mathcal{S}}\left[\frac{4\beta\zeta_2}{(\beta+\frac{p}{2})^2}\left(\int_\Omega w^{\left(\beta+\frac{p}{2}\right)\frac{N}{N-2}}\,dx\right)^{\frac{N-2}{N}}-\frac{4\beta\zeta_2}{(\beta+\frac{p}{2})^2}\int_\Omega w^{\beta+\frac{p}{2}}\,dx\right],
\end{multline*}
where $C_{\mathcal{S}}$ is the constant of the Sobolev embedding.
\end{proof}
We now handle the terms in the right-hand side of \eqref{vareq}. We start with the one involving the source $f$ of the equation. 
\begin{lemma}\label{rhs}
For any $\delta_1>0$ there exists $c_2>0$ depending on $\beta,N$ such that
\begin{equation}\label{l28}
-2\int_\Omega f\mathrm{div}(Du\ w^\beta)\,dx\leq \delta_1\int_\Omega |D^2 u|^2w^{\beta+\frac{p-2}{2}}\,dx+\frac{c_2}{\delta_1}\int_\Omega |f|^2w^{\beta+\frac{2-p}{2}}\,dx.
\end{equation}
\end{lemma}
\begin{proof}
Using \eqref{cz} and then the weighted Young's inequality we deduce that for any $\delta_1>0$ 
\begin{align*}
-2\int_\Omega f\mathrm{div}(Du w^\beta)\,dx&=-2\int_\Omega f\Delta u w^{\beta}\,dx-2\beta\int_\Omega fDu\cdot Dw w^{\beta-1}\,dx\\
&\leq 2\sqrt{N}\int_\Omega |f||D^2 u| w^{\beta}\,dx +4\beta\int_\Omega |f||D^2u|w^{\beta}\,dx\\
&=(4\beta+2\sqrt{N})\int_\Omega |f||D^2 u|w^{\frac{\beta}{2}+\frac{p-2}{4}}w^{\frac{\beta}{2}+\frac{2-p}{4}}\,dx
\end{align*}
and \eqref{l28} is a consequence of Young's inequality. 
\end{proof}
We now discuss the integral term involving $H$.
\begin{lemma}\label{Hphi}
There exist constants $c_3,c_4>0$ depending on $C_H$ such that 
\[
-\int_\Omega H(Du)\varphi\,dx\leq \frac{c_3}{\beta+1}\int_\Omega w^{\beta+\gamma+\frac{2-p}{2}}+\frac{c_4}{\beta+1}\int_\Omega |D^2u|^2w^{\beta+\frac{p-2}{2}}\,dx\ .
\]
\end{lemma}
\begin{proof}
We integrate by parts and use the boundary condition, together with the assumptions on $H$, to conclude
\begin{align*}
2\sum_{j=1}^N&\int_\Omega H(Du)\partial_{x_j}(\partial_{x_j}u w^\beta)\,dx\\
&=-2\sum_{j=1}^N\int_\Omega \partial_{x_j}H(Du)\partial_{x_j}u w^\beta\,dx
+2\int_{\partial\Omega}w^\beta H(Du)\partial_\nu u\,dS\\
&=-2\sum_{j=1}^N\int_\Omega \partial_{x_j}H(Du)\partial_{x_j}u w^\beta\,dx.
\end{align*}
Using once more the boundary condition $\partial_\nu u=0$  and noticing that 
$H_\xi(Du) \cdot \nu = \frac{h' ( |\xi|)  }{|\xi|} \xi \cdot \nu$, 
combined with the Cauchy-Schwarz and Young inequalities, we get
\begin{align*}
-2\sum_{j=1}^N\int_\Omega \partial_{x_j}H(Du)\partial_{x_j}u\ w^\beta\,dx&=-2\int_\Omega H_\xi(Du)\cdot Dw\ w^\beta\,dx
=-2\int_\Omega H_\xi(Du)\cdot D\left(\frac{w^{\beta+1}}{\beta+1}\right)\,dx\\
&=\frac{2}{\beta+1}\int_\Omega \mathrm{div}(H_\xi(Du))w^{\beta+1}\,dx
-\frac{2}{\beta+1}\int_{\partial \Omega} w^{\beta+1}H_\xi(Du)\cdot\nu\,dS
\\
&\leq\frac{2}{\beta+1}\int_\Omega |H_{\xi\xi}(Du)||D^2u|w^{\beta+1}\,dx
\leq \frac{2C_H}{\beta+1}\int_\Omega|D^2 u |w^{\frac{\gamma}{2}+\beta}\,dx\\
&\leq \frac{c_3}{\beta+1}\int_\Omega w^{\beta+\gamma+\frac{2-p}{2}}+\frac{c_4}{\beta+1}\int_\Omega |D^2u|^2w^{\beta+\frac{p-2}{2}}\,dx\ .
\end{align*}
\end{proof}
We plug the estimates in Lemmas \ref{diff1}, \ref{diff2}, \ref{diff3}, \ref{rhs}, \ref{Hphi} and choose $\delta_1=\frac{\zeta_1}{4}$ to deduce the following estimate. 
\begin{cor}Under the hypotheses of Theorem \ref{main1} we have that the following inequality holds true: 
\begin{multline*}
\zeta_3\left(\int_\Omega w^{\left(\beta+\frac{p}{2}\right)\frac{N}{N-2}}\,dx\right)^{\frac{N-2}{N}}+c_5\int_\Omega w^{\beta+\gamma+\frac{2-p}{2}}\,dx+\frac{\zeta_1}{4}\int_\Omega |D^2 u|^2w^{\beta+\frac{p-2}{2}}\,dx\\
\leq \frac{c_3}{\beta+1}\int_\Omega w^{\beta+\gamma+\frac{2-p}{2}}\,dx+\frac{c_4}{\beta+1}\int_\Omega |D^2u |^2w^{\beta+\frac{p-2}{2}}\,dx\\
+c_6\int_\Omega f^2w^{\beta+\frac{2-p}{2}}\,dx+\zeta_4\int_\Omega w^{\beta+\frac{p}{2}}\,dx\,,
\end{multline*}
where the constants that appear depend on the data of the problem, but not on $\eps$. 
\end{cor}
We are now ready to prove  Theorem \ref{main1}. 

\begin{proof}[Proof of Theorem \ref{main1}]

We first choose  $\beta$ sufficiently large to ensure the validity of the inequality
\[
\zeta_3\left(\int_\Omega w^{\left(\beta+\frac{p}{2}\right)\frac{N}{N-2}}\,dx\right)^{\frac{N-2}{N}}+\frac{c_5}2 \int w^{\beta+\gamma+\frac{2-p}{2}}\,dx
\leq c_6\int_\Omega f^2w^{\beta+\frac{2-p}{2}}\,dx+\zeta_4\int_\Omega w^{\beta+\frac{p}{2}}\,dx.
\]
We apply the H\"older's inequality and then the weighted Young's inequality (exploiting that $\gamma>p-1$) to the last term to find for any $\delta_2>0$
\[
\zeta_4\int_\Omega w^{\beta+\frac{p}{2}}\,dx
\leq \delta_2\int_\Omega w^{\beta+\gamma+\frac{2-p}{2}}\,dx+c_7,
\]
where $c_7$ depends on $\delta_2,\beta,\gamma,p,|\Omega|$ and blows-up  as $\delta_2\to0$.
Taking $\delta_2=\frac12 C_5$ we find through the H\"older's inequality
\begin{multline*}
\zeta_3\left(\int_\Omega w^{\left(\beta+\frac{p}{2}\right)\frac{N}{N-2}}\,dx\right)^{\frac{N-2}{N}}
\leq c_6\int_\Omega |f|^2w^{\beta+\frac{2-p}{2}}\,dx+c_7\\
\leq c_6\|f^2\|_{L^\alpha(\Omega)}\left(\int_\Omega w^{\left(\beta+\frac{p}{2}\right)\frac{N}{N-2}}\right)^{\frac{(2\beta+2-p)(N-2)}{(2\beta+p)N}}+c_7,
\end{multline*}
where \[
 \alpha'=\frac{(\beta+\frac{p}{2})\frac{N}{N-2}}{\beta+\frac{2-p}{2}}=\frac{(2\beta+p)N}{(2\beta+2-p)(N-2)} 
 \quad \mbox{and} \qquad \alpha=\frac{N(2 \beta+p) }{4\beta+ 2 (p-1)N-2(p-2)} \, .
 \]
 Hence, for any $\beta $ large enough we have that 
 \[
\big\|w\big\|_{L^{ (\beta+\frac{p}{2} )\frac{N}{N-2}}(\Omega)}^{\beta+\frac{p}{2}} 
\leq  c_8 \bigg( \big\|f \big\|^2_{L^{2\alpha}(\Omega)} 
\big\|w\big\|_{L^{ (\beta+\frac{p}{2} )\frac{N}{N-2}}(\Omega)}^{\beta+1-\frac{p}{2}} 
+1 \bigg)\,, 
\]
and we observe  that as $\beta\to\infty$,  $\alpha'\to\frac{N}{N-2}$ and  $\alpha\nearrow \frac{N}{2}$. 
Since $ {\beta+1-\frac{p}2 } <{\beta+\frac{p}2} $, we can further apply the weighted Young inequality and find
\[
\big\|w\big\|_{L^{ (\beta+\frac{p}{2} )\frac{N}{N-2}}(\Omega)}^{\beta+\frac{p}{2}} 
\leq  c_9 \bigg( \big\|f \big\|^{\frac{2\beta+p}{p-1}}_{L^{2\alpha}(\Omega)} 
+1 \bigg)\,, 
\]
deducing \eqref{est1} with $\eta= (\beta+\frac{p}{2} )\frac{N}{N-2}$ and $q_\eta = 2 \alpha$. 

  \end{proof}
  
{\begin{rem}[On the convexity of the domain]\label{conv} Arguing as in \cite{LPcpde}, the convexity assumption in Theorem \ref{main1} can be removed by considering the equation satisfied by $z(x)=|Du(x)|^2e^{\eta d(x)}$ (instead of just $w=|Du|^2$), where $d(x)$ is a $C^2$ positive function in $\Omega$ that coincides with the distance function in a neighborhood of the boundary. Indeed, if $\eta$ is chosen such that $\eta\geq \|(D^2d)_+\|_{L^\infty(\partial\Omega)}$, one has $\partial_\nu z\leq 0$ on $\partial \Omega$, see \cite[Lemma 2.3]{PorrCCM}. We briefly show how to handle the extra terms in the simple case $a(t) \equiv 1$, i.e. for  
\[
\begin{cases}
-\Delta u+|Du|^\gamma=f(x)\qquad &\text{ in }\Omega\,,\\
\partial_\nu u=0 & \text{ on }\partial\Omega.
\end{cases} 
\]
First, we note that since  $w=|Du|^2$ solves
\[
-\Delta w+2|D^2u|^2+\gamma|Du|^{\gamma-2}Du\cdot Dw=2Df\cdot Du\quad \mbox{ in } \Omega, 
\]
then $z$ satisfies  
\begin{multline*}
-\Delta z+2e^{\eta d(x)} |D^2u|^2+\gamma|Du|^{\gamma-2}Du\cdot Dz
\\
 = 2e^{\eta d(x)} Df\cdot Du+\gamma \eta |Du|^{\gamma-2}(Du\cdot D d) z-2 \eta Dz\cdot  D d + c_\eta z \quad \mbox{ in } \Omega. 
\end{multline*}
Observe that  new terms are the last three in the right-hand side above. 
Having at hand the proof of Theorem  \ref{main1}, it follows that such  terms can be, roughly speaking,   absorbed by the superlinear term coming from $|D^2u| $, after one has plugged in the equation for $u$ and has exploited the growth condition \eqref{H1}. 
\end{rem}

  \begin{rem}  If $q > N$, one can get better estimates when the diffusion is driven by the $p$-Laplacian $\Delta_p$, $p>1$ (and for more general quasi-linear equations modelled on these operators). Indeed, if $f\in L^q(\Omega)$, $q>N$, it follows that $Du$ is controlled in $L^r$ for all finite $r>1$, and hence for $r$ sufficiently large, we have $|Du|^\gamma\in L^q$, $q>N$, so that the equation can be regarded as $-\Delta_p u=-|Du|^\gamma+f$, with right-hand side bounded in $L^q$, $q>N$. This leads to Lipschitz estimates through the results for the $p$-Poisson equation in e.g. \cite[Theorem 4.3]{CianchiMazyaJEMS}, \cite[Theorem 3.1]{CianchiMazyaCpaa}, see also \cite{CianchiMazyaCPDE,DP}. Actually, by bootstrapping one gets also $C^{1,\alpha}$ bounds, see e.g. \cite{DibNA,DiB}.
  \end{rem}
  
  \begin{rem}
  The assumption $f\in L^N$ is in general not sufficient to obtain gradient boudedness, neither for the Poisson equation. Sharp assumptions in Lorentz classes have been found in various works, see \cite{BeckMingione} and the references therein, while we refer to \cite{Cianchi92} for questions related to the optimality of such a condition. 
  \end{rem}
  
    \begin{rem}[Low dimensions]  When $N=1,2$ the result can be obtained on the same way, as one can exploit the continuous embedding of $W^{1,2}$ into $L^s$ for any finite $s>1$.
  \end{rem}

 \begin{rem}\label{remu}   The result  of  Theorem \ref{main1} is still valid if one adds a zero-th order term $\lambda u$, with $\lambda>0$, in the equation. In this case the estimate   will depend on $\|f-\lambda u\|_q$ instead of $\|f \|_q$. 
  \end{rem}

  \section{Maximal regularity estimates for $p \ge 2$}\label{s_max}
  Let us now consider a function $\tilde{a}:[0,\infty)\to[0,\infty)$ of class $C^1(0,\infty)$, and the problem
 \begin{equation*}
\begin{cases}
\lambda u  -\mathrm{div}(\tilde a(|D u|)D u) + H(Du)=f(x)&\text{ in } \Omega,\\
\partial_\nu u=0&\text{ on } \partial\Omega,
\end{cases}
\end{equation*}
which is the same as the one of the previous section, but written for convenience with the slightly different notation $\tilde a(t) = a(t^2)$. Here, we assume that there exist $p\ge2 $ and constants $\bar c_{\tilde a}, \bar C_{\tilde a},C_{\tilde{a}}>0$  such that 
\begin{equation}\tag{$ \widetilde{ A1 }$}\label{At}
0\leq\  \inf_{t>0}\frac{t\tilde{a}'  (t)}{\tilde{a}  (t)} \ \leq \ \sup_{t>0}\frac{t\tilde{a}'  (t)}{\tilde{a}  (t)}\ \leq \ C_{\tilde{a}}\ < \ \infty\,,
\end{equation}
and
\begin{equation}\tag{$\widetilde{A2}$}\label{A1t}
\bar c_{\tilde a} t^{{p-2}}\leq \tilde{a}  (t)\leq \bar C_{\tilde a} t^{ {p-2}} \,, \quad \forall t >0.
\end{equation}
Note that the above hypotheses  are nothing but  \eqref{Aa} and \eqref{A1}, with constants possibly differing by a factor of two, and with the additional constraint $p \ge 2$.

As far as  $H$ is concerned, we assume here that
\begin{equation}\tag{$\widetilde{H1}$}\label{Ht}
H\in C^0 ( \R^N\setminus\{0\})
\end{equation}
and the existence of $k_1>0 $, $c_H$, $C_H>0$ and $ \gamma>1$   such that  
\begin{equation}\tag{$\widetilde{H2}$}\label{H1t}
c_H|\xi|^\gamma \leq H(\xi) \quad \mbox{ and } \quad  |H_\xi(\xi)|\leq C_H|\xi|^{\gamma-1} 
\qquad \forall \xi \in \R^N \, : \quad |\xi|\geq k_1 \,.
\end{equation}
Note that for any $0 \le \eps \le 1$, we have as a consequence that
\begin{equation}\label{Hbelow}
\frac{c_H}2(\eps+|\xi|^2)^{\frac{\gamma}{2}} - c_\eps \leq H(\xi) 
\qquad \forall \xi \in \R^N \, : \quad |\xi|\geq k_1 \,.
\end{equation}
for some $c_\eps$ that vanishes as $\eps \to 0$.
   
 \begin{thm}\label{main2}  
 Let $\Omega\subset \R^N$, $N\geq 3$, be a $C^2$ convex domain  and assume that   \eqref{At},\eqref{A1t},\eqref{Ht},\eqref{H1t} are in force. Suppose that 
 $\gamma>p-1$,  $p\geq 2$,  $\lambda\geq0$,  and 
  $$f \in L^q(\Omega) \qquad \text{with } \qquad  q > \max\left\{ \frac{N(\gamma-(p-1))}{\gamma}, 2 \right\}
  ,$$ 
  
  then  there exists $K$ depending on $\|Du\|_{L^1(\Omega)}, \|f - \lambda u\|_{L^q(\Omega)}, q, N, p, \gamma$ and the constants in the standing assumptions such that any smooth solution $u_\eps$ to 
  \begin{equation*}
\begin{cases}
\lambda u -\mathrm{div} \big(\tilde a (\sqrt{ |D u|^2+\eps} ) \ D u\big) + H(Du)=f_\eps(x)&\text{ in } \Omega,\\
\partial_\nu u=0&\text{ on } \partial\Omega .
\end{cases}
\end{equation*}
 satisfies 
\begin{equation} \label{estimate}
 \| Du_\eps \|_{L^{q\gamma}(\Omega)} \le K.
\end{equation}

Moreover if  
$$q = \frac{N(\gamma-(p-1))}{\gamma} \quad \mbox{ and } \quad \gamma >\frac{N(p-1)}{N-2}
$$ then, \eqref{estimate} holds true
  \begin{itemize}
  \item[(i)]  provided that $\|f\|_{L^q(\Omega)}$ is small enough if $\lambda=0$;    
  \item[(ii)] if $\lambda>0$,  with $K$  depending also on $\lambda, \|u\|_{L^q(\Omega)}$ and it remains bounded whenever $f$ varies in a set of $L^q(\Omega)$-uniformly integrable functions. 
  \end{itemize}
  \end{thm}

  \begin{rem}
  The linear case $\tilde a(t) \equiv \tilde a > 0$ has been already covered in \cite{CGell} in the periodic setting and later in \cite{GP} for Neumann problems when $\lambda=0$. Here, we are also able to deal with the borderline integrability exponent $q = \frac{N(\gamma-1)}{\gamma}$. For such value of $q$ the maximal regularity result for linear diffusions is new when $\gamma\geq2$. The case (ii) in the regime $\gamma<2$ is contained in \cite{G}. 
    \end{rem}
  \begin{rem}
  Note that $N\frac{\gamma-(p-1)}{\gamma}>1$ whenever $\gamma>\frac{N(p-1)}{N-1}$, so one has always to restrict at least to this range of growth for the first-order term. 

  \end{rem}
\begin{rem}
Once $|Du|^\gamma$ is controlled in $L^q$ by means of Theorem \ref{main2}, one can regard the equation as $\lambda u  -\mathrm{div}(\tilde a(|D u|)D u) =-|Du|^\gamma+f$, hence $Du$ is  bounded in $L^{\frac{Nq(p-1)}{N-q}}$ by the results in \cite{CianchiMazyaJEMS}. However, it is to our knowledge an open problem the validity of the stronger bound   $|Du|^{p-2}Du\in W^{1,q}$.  
This has been shown to hold in the case $q=2$ in the recent work \cite{CianchiMazyaARMA2}, which by the way is not covered by Theorem \ref{main2}.
\end{rem}

 \begin{proof}
As in the previous case, since our method needs to deal with smooth solutions, we need to first approximate our problem. 
Let us consider the sequence of solutions to 
  \begin{equation} \label{pb}
  \begin{cases}
\lambda u_\eps   - \mathrm{div} \big(\tilde{a} (v_\eps)  Du_\eps\big)+ H(Du_\eps) =f_\eps(x)&\text{ in } \Omega,\\
\partial_\nu u_\eps=0&\text{ on } \partial\Omega\,. 
\end{cases}
  \end{equation}
 where $f_\eps$ is a suitable regularization of $f$ and $ v_\eps=\sqrt{|Du_\eps|^2+\eps}$.

 Let us now set  $v_{k} = v_{\eps, k}=   (\sqrt{|Du_\eps |^2+\eps} -k)^+   $, for any $k\ge k_1$. Moreover, we denote by $\Omega_{k}=\{x\in\Omega: v_\eps >k\}$ and observe that
 \begin{equation*}
 \partial \Omega_{k}=\{x\in\Omega: v_\eps=k\}\cup(\partial\Omega\cap \overline{\Omega}_{k}).
 \end{equation*}
We also preliminary note that
 \begin{equation*}
Dv_{\eps,k}=Dv_\eps\text{ on }\Omega_k   \,, \quad Dv_{\eps,k}=0\text{ on } \Omega\setminus\Omega_k,\qquad \mbox{ and }  \quad 
v_{\eps,k}=0\text{ on }\partial\Omega_k\cap\Omega.
\end{equation*}
Moreover, by Lemma \ref{signw} we have, thanks to the convexity of $\Omega$,  
 \begin{equation}\label{nuvk}
 \partial_\nu v_\eps=\frac{\partial_\nu |Du_\eps|^2}{2v_\eps}\leq 0 \quad \text{ on } \quad \partial\Omega.
 \end{equation}
 
We may also add $c_\eps$ to both sides of the equation, so that, by \eqref{Hbelow}, $H$ satisfies
 \begin{equation}\label{Hbelow2}
\frac{c_H}2(\eps+|\xi|^2)^{\frac{\gamma}{2}} \leq H(\xi),  \qquad \forall \xi \in \R^N \setminus \{0\}\,.
\end{equation}
 
 \medskip

 The result will be a consequence of the following property, that will be shown below: there exist $k_0, c, \omega > 0$ depending on the data and the constants appearing in the assumptions such that $\omega$ is small enough so that the inequality
\[
Z^{\frac{N-2}{N}} < \omega + c Z
\]
is true if and only if $Z \in [0, Z^-) \cup ( Z^+, +\infty)$.  
Our aim is to apply such an inequality to $Z= \|v_{k} \|_{L^{q \gamma} (\Omega)}$, i.e. we are done if we are able to prove that 
 \begin{equation}\label{crucialest}
\exists k_0 >0 \, : \quad \forall k \geq k_0 \qquad 
\left(\int_\Omega \left(\left(v_\eps -k\right)^+\right)^{r \gamma}\,dx\right)^{\frac{N-2}{N}}\leq \omega_k + c \int_\Omega \left(\left(v_\eps -k\right)^+\right)^{r \gamma}\,dx, 
 \end{equation}
with $\omega_k \to 0$ as $k$ diverges. 

 Indeed, once \eqref{crucialest} is established, one can conclude as follows (see  \cite[Section 2, p.1524-1525]{CGell}): since $k \mapsto Z(k) := \int_\Omega \left(\left(v_\eps -k\right)^+\right)^{r \gamma}$ is continuous and it vanishes  as $k\to \infty$, then $Z(k) \le Z^-$ for all $k \ge k_0$, and therefore
 \[
\|v_\eps\|^{r\gamma}_{L^{r\gamma}(\Omega)} \le \|(v_\eps-k_0)^+\|^{r\gamma}_{L^{r\gamma}(\Omega)} + \|k_0\|^{r\gamma}_{L^{r\gamma}(\Omega)} \le Z^- + |\Omega| k_0^{r\gamma}.
 \]
 
 \medskip

   From now on we drop the subscript $\eps$ for brevity, and look of course for estimates which are independent from $\eps$.

To get \eqref{crucialest}, we start by testing the equation in \eqref{pb} by  
$$\varphi=   \mathrm{div}\bigg(D u\ \frac{ v_{k}^{\beta}}{v}\bigg)   \quad \mbox{ for some } \ \beta >1\,,
$$ 
($\beta$ to be  determined later) and  integrating by parts. 
Exploiting that $u$ satisfies the homogeneous Neumann boundary condition, we 
have 
\begin{equation}\label{weak}
- \int_{\Omega} \tilde{a} (v )  Du  \cdot D\varphi \ dx =  
\int_{\Omega} \big( H(Du)  + \lambda u -f(x)  \big) \varphi \ dx \,.
  \end{equation}
 
We first elaborate on the left-hand side of the previous equality.
 
 \begin{lemma}\label{t2s1} We have the following inequality
 \[
 - \int_{\Omega} \tilde{a} (v )  Du  \cdot D\varphi \ dx \ge \int_{\Omega_k} \tilde a(v) \frac{v_k^\beta}v |D^2u|^2\,dx+\bar{c}_{\tilde a}(\beta-1) \int_{\Omega_k}  v^{p-2}   v_{k}^{\beta-1}  \ 
   |D v_k|^2  dx,
 \]
 where $\bar c_{\tilde a}$ is the constant appearing in \eqref{A1t}.
 \end{lemma} 
 
 \begin{proof} 
 Using the boundary condition $\partial_{\nu} u =0$ on $\partial \Omega$, we deduce that the left-hand side of the above inequality is equal to  
   \begin{multline*}
- \int_{\Omega} \tilde{a} (v) Du\cdot D\varphi\,dx
 = - \sum_{i,j=1}^N \int_{\Omega} \tilde{a} (v)  \ \partial_{x_i}  u\  \partial_{x_i}\bigg(\partial_{x_j}\Big(\partial_{x_j}u\ \frac{v_{k}^{\beta}}{v} \ \Big)\bigg)\,dx\\
 = - \sum_{i,j=1}^N\int_{\Omega}  \tilde{a} (v)  \ \partial_{x_i}  u \  \partial_{x_j}\bigg(\partial_{x_i} \Big(\partial_{x_j}u\ \frac{v_{k}^{\beta}}{v} \ \Big)\bigg)\,dx
    \end{multline*}
       \begin{multline*}
 = \sum_{i,j=1}^N\int_{\Omega} \partial_{x_j}\big(\tilde{a} (v)  \ \partial_{x_i} u\big) \partial_{x_i} \bigg(\partial_{x_j}u\ \frac{v_{k}^{\beta}}{v} \ \bigg)\,dx
  - \sum_{i,j=1}^N  \int_{\partial \Omega } \tilde{a} (v)  \ \partial_{x_i} u \   \partial_{x_i}\bigg(\partial_{x_j} u\ \frac{v_{k}^{\beta}}{v} \  \bigg) \nu_j \,dS\\
  : = I_1 + I_2\,.
  \end{multline*}

 We first handle the boundary terms as follows
\begin{align*} 
   I_2 &= -  \sum_{i,j=1}^N  \int_{\partial \Omega} \tilde{a} (v)   \ \partial_{x_i} u \    \partial_{x_i}\bigg(\frac{v_{k}^{\beta}}{v} \  \bigg) \big( \partial_{x_j} u \ \nu_j \big) \,dS
-
\sum_{i,j=1}^N  \int_{\partial\Omega} \tilde{a} (v)    \frac{v_{k}^{\beta}}{v} \     \bigg(\partial_{x_j  x_i} u\  \partial_{x_i} u  \ \nu_j  \ 	 \bigg)\,dS
\\
&=  -   \int_{\partial  \Omega} \tilde{a} (v)  \   D u  \cdot  D \big(\frac{v_{k}^{\beta}}{v} \  \big)   \partial_{\nu} u   \,dS
-
\int_{\partial \Omega} \tilde{a} (v)    \  v_{k}^{\beta}    \ \partial_{\nu} v   \ \,dS\geq 0 \,,
\end{align*}
where we have used both the boundary condition on $u$ and inequality \eqref{nuvk}.

Then, since $v_k$ vanishes outside $\Omega_k$, we get 
\begin{align*}
I_1&= \sum_{i,j=1}^N\int_{\Omega_k} \partial_{x_j}\big(\tilde{a} (v)  \ \partial_{x_i} u\big) \partial_{x_i} \bigg(\partial_{x_j}u\ \frac{v_{k}^{\beta}}{v} \ \bigg)\,dx \\
&= \sum_{i,j=1}^N \int_{\Omega_k} \left( \tilde{a}' (v) \ \partial_{x_j}v \ \partial_{x_i}u+\tilde{a} (v)  \ \partial_{x_ix_j}u\right)\left(\frac{v_k^\beta}{v}\ \partial_{x_ix_j}u+\frac{\beta v v_k^{\beta-1} \partial_{x_i}v_k- v_k^\beta \partial_{x_i}v }{v^2} \ \partial_{x_j}u \right) \,dx.
\end{align*}
As $v\geq v_k$, we have in particular that 
\[
\beta v^{-1}v_k^{\beta-1}-v^{-2}v_k^\beta\geq (\beta-1)\frac{v_k^{\beta-1}}{v}, 
\]
and therefore, owing to the fact that $a' (s) \geq 0 $  (see \eqref{At}) and since $\beta> 1$, we get
\[
\sum_{i,j=1}^N\int_{\Omega_k} \tilde{a}' (v) \frac{\beta v v_k^{\beta-1} - v_k^\beta}{v^2}  \ 
 \partial_{x_j}v\partial_{x_i}u\partial_{x_j}u \partial_{x_i}v_k \,dx
\geq (\beta-1)\int_{\Omega_k}\tilde{a}' (v) \frac{v_k^{\beta-1}}{v} \ (Dv\cdot Du)^2 \,dx \geq0,
\]
and also
\[
\sum_{i,j=1}^N\int_{\Omega_k}\tilde{a}' (v) \frac{v_k^\beta}{v} \ \partial_{x_j}v\partial_{x_i}u\partial_{x_ix_j}u \,dx =
\sum_{i,j=1}^N\int_{\Omega_k}\tilde{a}' (v) {v_k^\beta}  |Dv|^2 \,dx
\geq0\ .
\]
Thus, we are left with
\begin{align*}
I_1&\geq \sum_{i,j=1}^N\int_{\Omega_k}\tilde{a} (v)  \, \partial_{x_ix_j}u\left(\partial_{x_ix_j}u\frac{v_k^\beta}{v}+\partial_{x_j}u\frac{\beta v v_k^{\beta-1} \partial_{x_i}v_k-\partial_{x_i}v v_k^\beta}{v^2} \right) \,dx
\\&\geq \sum_{i,j=1}^N\int_{\Omega_k}\tilde{a} (v)  \,  \partial_{x_ix_j}u\left(\partial_{x_ix_j}u\frac{v_k^\beta}{v}+(\beta-1) \partial_{x_j}u\partial_{x_i}v\frac{v_k^{\beta-1}}{v}\right) \,dx \\
&=\int_{\Omega_k}\tilde{a} (v)  \, |D^2u|^2\frac{v_k^\beta}{v} \,dx+(\beta-1) \int_{\Omega_k}  \tilde{a} (v)  \,  v_{k}^{\beta-1}  \ 
   |D v_k|^2  dx \\
   &\geq \int_{\Omega_k}\tilde{a} (v)  \, |D^2u|^2\frac{v_k^\beta}{v} \,dx+\bar c_{\tilde a}(\beta-1) \int_{\Omega_k}  v^{p-2}   v_{k}^{\beta-1}  \ 
   |D v_k|^2  dx,
\end{align*}
where we used  \eqref{A1t} in the last inequality.
\end{proof}

\begin{lemma}\label{t2s2}There exist constants $c_{10},c_{11} > 0$ depending on $\bar c_{\tilde a}, \bar C_{\tilde a}, c_H,  N, C_{\tilde{a}}, \lambda$ such that the following inequality holds
\begin{equation}\label{l36}
\int_{\Omega_k}\tilde{a}(v)|D^2u|^2\frac{v_k^\beta}{v}\,dx\geq
c_{10}\int_{\Omega_k}v^{2\gamma+1-p}v_k^\beta\,dx-  
c_{11}\int_{\Omega_k}(\lambda u-f)^2v_k^\beta v^{1-p}\,dx
\end{equation}
\end{lemma}

\begin{proof}
 Note first that
\begin{multline*}
\mathrm{div}(\tilde{a}(v)Du)=\tilde{a}(v)\Delta u+\tilde{a}'(v)\frac{D^2uDu}{v}\cdot Du=\tilde{a}(v)\Delta u+\tilde{a}'(v)\frac{\Delta_\infty u}{v}
=\tilde{a}(v)\underbrace{\left[\Delta u+\frac{\tilde a'(v)}{v\tilde a(v)}\Delta_\infty u\right]}_{:= \tilde{\mathcal{A}}(D^2u)}.
\end{multline*}
Therefore, as before we define $\nu =\sqrt{N}+C_{\tilde{a}}$ and we have 
\[
|\tilde{\mathcal{A}}(D^2u)|=\left|\Delta u+\frac{\tilde{a}'(v)}{v\tilde{a}(v)}D^2uDuDu\right|\leq \sqrt{N}|D^2u|+C_{\tilde{a}}\frac{|D^2u||Du|^2}{\eps+|Du|^2}\leq \nu  |D^2u|.
\]
We exploit, as in Theorem \ref{main1}, the non-divergence formulation of the equation, i.e. we consider 
\[\tilde{a}(v)\tilde{\mathcal{A}}(D^2u)=H(Du)+ \big( \lambda u-f \big)  \quad \mbox{ in } \Omega\,,\] 
and  plugging the equation back in the term $\tilde{a}(v)|D^2u|^2$  
by  \eqref{At} and  \eqref{Hbelow2},   we get 
\begin{multline*}
\tilde{a}(v)|D^2u|^2\geq \frac{[\tilde{a}(v)\tilde{\mathcal{A}}(D^2u)]^2}{\nu^2 \tilde{a}(v)}\\
\geq \frac{1}{\nu^ 2 C_2}\left[H(Du)+\lambda u-f\right]^2v^{2-p}\geq \frac{1}{\nu^ 2C_2}\left[\frac{H^2(Du)}{2}- 2 (\lambda u- f)^2\right]v^{2-p}\\
= \frac{c_H^2}{\nu^2 C_2}\left\{\frac{v^{2\gamma}}{8}-2 (\lambda u- f)^2 \right\}v^{2-p}, 
\end{multline*}
and thus \eqref{l36} holds. 
\end{proof}

\begin{lemma}\label{t2s3} There exist constants $c_{12}, c_{13} > 0$ depending on $\bar c_{\tilde a}, \bar C_{\tilde a}, c_H,  C_H, N, C_{\tilde{a}}, \lambda, \beta,p $ such that the following inequality holds true:
\begin{align*}
 - \int_{\Omega} \tilde{a} (v )  Du  \cdot D\varphi \ dx   \ge c_{12} \int_{\Omega_k} v^{p-3 } v_k^\beta |D^2u|^2\,dx + c_{12} \left(\int_{\Omega_k} v_k^{(p+\beta-1)\frac{N}{N-2}}\right)^{\frac{N-2}{N}} \\ 
  +c_{12}\int_{\Omega_k} v^{p-2}   v_{k}^{\beta-1}  \ |D v_k|^2  dx + c_{12}\int_{\Omega_k}v^{2\gamma+1-p}v_k^\beta\,dx  
-c_{13}\int_{\Omega_k}2 (\lambda u- f)^2 v_k^\beta v^{1-p}\,dx .
\end{align*}
\end{lemma}

\begin{proof} Fist,  we exploit the Sobolev inequality on the term $\int_{\Omega_k}  v^{p-2}   v_{k}^{\beta-1} |D v_k|^2  dx$, and the fact that $v_k\leq v$ (and $p \ge 2$) to get
\begin{multline*}
\int_{\Omega_k}  v^{p-2}   v_{k}^{\beta-1} |D v_k|^2  dx\geq  \int_{\Omega_k}  v_{k}^{\beta+p-3} |D v_k|^2  dx=\frac{4}{(p+\beta-1)^2}\int_{\Omega_k}|Dv_k^{\frac{p+\beta-1}{2}}|^2\\
\geq \frac{4C_{\mathcal{S}}}{(p+\beta-1)^2}\left(\int_{\Omega_k} v_k^{(p+\beta-1)\frac{N}{N-2}}\right)^{\frac{N-2}{N}}-\frac{4}{(p+\beta-1)^2} \int_{\Omega_k} v_k^{p+\beta-1}\,dx,
\end{multline*}
where $C_{\mathcal{S}}$ is the constant of the Sobolev embedding. To conclude, it now suffices to use Lemma \ref{t2s1} and Lemma \ref{t2s2}.
\end{proof}

We now work on the right-hand side of \eqref{weak}, that is,
\begin{multline*}
\int_{\Omega_k} \big( H(Du)  + \lambda u -f  \big) \mathrm{div}\Big(D u\ \frac{ v_{k}^{\beta}}{v}\Big) \ dx\\
=\underbrace{\int_{\Omega_k}  H(Du)\mathrm{div}\Big(D u\ \frac{ v_{k}^{\beta}}{v}\Big)}_{:= I_3}
+
\underbrace{ \lambda  \int_{\Omega_k}u\ \mathrm{div}\Big(D u\ \frac{ v_{k}^{\beta}}{v}\Big)}_{:=I_4}- \underbrace{\int_{\Omega_k}f\ \mathrm{div}\Big(D u\ \frac{ v_{k}^{\beta}}{v}\Big)}_{:=I_5}.
\end{multline*}

\begin{lemma}\label{t2s4} For any $\delta > 0$, there exists $c_{14} > 0$ depending on $\delta, c_H, C_H,\gamma,p,N$ such that 
\begin{align*}
I_3 + I_4 - I_5   \le -\lambda \int_{\Omega_k}|Du|^2\frac{ v_{k}^{\beta}}{v}\,dx  
  + \delta \int_{\Omega_k} v^{p-2} v_{k}^{\beta-1} |Dv_k|^2 \,dx + \delta \int_{\Omega_k} v^{2\gamma+1 -p} v_{k}^{\beta}\,dx \\
  + \delta \int_{\Omega_k} |D^2u|^2v^{p-3}v_{k}^\beta\,dx + \delta \int_{\Omega_k} v_{k}^{p+\beta-3}|Dv_k|^2\,dx 
  + c_{14} \int_{\Omega_k}v_{k}^{\beta+2\gamma+1 - p }\,dx + c_{14} \int_{\Omega_k} |f|^2v_{k}^{\beta+ 1-p}\,dx.
\end{align*}

\end{lemma}

\begin{proof}
We start with $I_3$ above, and use the divergence theorem, the boundary condition $\partial_\nu u=0$ at $\partial\Omega$, \eqref{H1t} 
to find that
\begin{multline*}
\int_{\Omega}  H(Du)\mathrm{div}\Big(D u\ \frac{ v_{k}^{\beta}}{v}\Big)\,dx=-\int_{\Omega}\partial_{x_j}(H(Du))\partial_{x_j}u\ \frac{ v_{k}^{\beta}}{v}\,dx+\int_{\partial\Omega}H(Du)\ \frac{ v_{k}^{\beta}}{v}\partial_\nu u\,dS\\
=-\int_{\Omega}H_\xi(Du)D^2uDu\frac{ v_{k}^{\beta}}{v}\,dx\leq \int_{\Omega_k}|H_\xi(Du)||Dv|v_{k}^{\beta}\,dx\leq C_H\int_{\Omega_k} v^{\gamma-1}|Dv|v_k^\beta \, dx.
\end{multline*}
We now set
\[
\eta=2\gamma-p+1
\]
and observe that $\eta>1$ since $p\geq2$ and $\gamma>p-1$. Thus, for any $\delta>0$ we have

\[
 C_H\int_{\Omega_k} v^{\gamma-1}|Dv|v_k^\beta\,dx \leq \delta\int_{\Omega_k} v^{p-2}|Dv|^2v_{k}^{\beta-1}\,dx+\frac{C_H^2}{4\delta}\int_{\Omega_k} v^{2\gamma-p}v_{k}^{\beta+1}\,dx.
\]
    Moreover, applying  the weighted Young's inequality to the last term, we have that 
    \[
    \frac{C_H^2}{ 4\delta}\int_{\Omega_k} v^{\eta-1}v_{k}^{\frac{\beta}{\eta'}}v_{k}^{\frac{\beta}{\eta}+1}\,dx\leq \delta \int_{\Omega_k} v^\eta v_{k}^{\beta}\,dx+C(C_H,\gamma,p,N,\delta)\int_{\Omega_k} v_{k}^{\beta+\eta}\,dx\,,
    \]
and thus 
 \[
I_3 \leq \delta \int_{\Omega_k} v^{p-2}|Dv|^2v_{k}^{\beta-1}\,dx+ \delta \int_{\Omega_k} v^\eta v_{k}^{\beta}\,dx+C(c_H,\gamma,p,N,\delta) \int_{\Omega}v_{k}^{\beta+\eta}\,dx.
 \]
 
 \medskip 
 
In order to handle  $I_4$   we apply  the divergence theorem and we exploit the Neumann boundary condition, so that 
\[
I_4= \lambda \int_{\Omega}u \ \mathrm{div}\Big(D u\ \frac{ v_{k}^{\beta}}{v}\Big)=-\lambda \int_{\Omega}|Du|^2\frac{ v_{k}^{\beta}}{v}\,dx+\lambda\int_{\partial\Omega}\frac{ uv_{k}^{\beta}}{v}\partial_\nu u\,dS=-\lambda \int_{\Omega_k}|Du|^2\frac{ v_{k}^{\beta}}{v}\,dx.
\]

 \medskip 
 We finally consider the term $I_5$ involving the source $f$ of the equation, which becomes, after integration by parts (recall also that $v_k \le v$)
    \[
|I_5| =     \left|\int_{\Omega} f\mathrm{div}\left(\frac{v_{k}^\beta}{v} Du\right)\,dx\right|\leq \int_{\Omega_k} |f|\big[v^{-1}v_{k}^\beta\sqrt{N}|D^2u|+(\beta+1)|Dv_{k}|v_{k}^{\beta-1} \big]\,dx.
    \]
    Then, the Young's inequality gives for any positive   $\delta$
    \[
    \int_{\Omega_k} |f|v^{-1}v_{k}^\beta\sqrt{N}|D^2u|\,dx\leq \delta \int_{\Omega_k} |D^2u|^2v^{p-3}v_{k}^\beta\,dx+\frac{C}{\delta}\int_{\Omega_k} |f|^2v^{1-p}v_{k}^\beta\,dx\,,
    \]
and moreover 
    \[
    (\beta+1)\int_{\Omega_k} |f||Dv_{k}|v_{k}^{\beta-1}\,dx\leq \delta \int_{\Omega_k} v_{k}^{p+\beta-3}|Dv_k|^2\,dx+\frac{C}{\delta}\int_{\Omega_k} |f|^2v_{k}^{1-p+\beta}\,dx.
    \]
    Since $v\geq v_{k}$ and $p>1$, it follows that $v^{1-p}\leq v_{k}^{1-p}$, and collecting all the above estimates we are done.
     \end{proof}

\begin{lemma}
There exists $c_{15} > 0$ depending on $\bar c_{\tilde a}, \bar C_{\tilde a}, C_H, c_H, N, C_{\tilde{a}}, \lambda, \beta,p, \gamma$ such that, for any $r>2$ 
\begin{multline}\label{mainineq}
c_{15} \left(\int_{\Omega_k} v_k^{(p+\beta-1)\frac{N}{N-2}}\,dx\right)^{\frac{N-2}{N}} + \lambda c_{15} \int_{\Omega_k}|Du|^2\frac{ v_{k}^{\beta}}{v}\,dx 
\\ 
\le  \int_{\Omega_k} |f|^r\,dx+  \int_{\Omega_k} v_{k}^{(\beta + 1-p)\frac{r}{r-2}}\,dx 
+ \lambda^2 \int_{\Omega_k} |u|^2 v_k^\beta v^{1-p}\,dx  +  \int_{\Omega_k} v_k^{p+\beta-1}\, dx+
 \int_{\Omega_k}v_{k}^{\beta+2\gamma - p + 1}\,dx.
\end{multline}
\end{lemma}

\begin{proof}
Combining Lemma \ref{t2s3} and \ref{t2s4} with $\delta$ small enough (with respect to $c_4$) we get
\begin{multline}\label{ci1}
c_{12} \left(\int_{\Omega_k} v_k^{(p+\beta-1)\frac{N}{N-2}}\,dx\right)^{\frac{N-2}{N}} + \lambda \int_{\Omega_k}|Du|^2\frac{ v_{k}^{\beta}}{v}\,dx  \le 
c_{14} \int_{\Omega_k}v_{k}^{\beta+2\gamma - p + 1}\,dx 
\\
+
  c_{14}  \int_{\Omega_k} |f|^2v_{k}^{1-p+\beta}\,dx 
+ c_{12} \int_{\Omega_k} v_k^{p+\beta-1}
+  c_{14}\int_{\Omega_k} 2(\lambda u- f)^2 v_k^\beta v^{1-p}\,dx .
\end{multline}
To handle the  terms involving the datum $f$, we exploit that $v_k \leq v$, the inequality $(a+b)^2\leq 2(a^2+b^2)$, $a,b\geq0$, and we finally employ Young's inequality, so that 
\begin{multline*}
\int_{\Omega_k} (\lambda u- f)^2 v_k^\beta v^{1-p}\,dx  
+ \int_{\Omega_k} |f|^2v_{k}^{1-p+\beta}\,dx
\\
\leq c_{16}\left( \int_{\Omega_k} |f|^r\,dx 
+ \int_{\Omega_k} v_{k}^{(1-p+\beta)\frac{r}{r-2}}\,dx
+  \lambda^2 \int_{\Omega_k} |u|^2 v_k^\beta v^{1-p}\,dx \right) 
,
\end{multline*}
and \eqref{mainineq} follows.
\end{proof}

 Now we are ready to prove Theorem \ref{main2}.

 \proof[Proof Theorem \ref{main2}]
 
In view of the discussion at the beginning of this section, we show \eqref{crucialest}. 

\medskip

$\bullet$ We start supposing that  
$$q > \frac{N(\gamma-(p-1))}{\gamma} \qquad \mbox{ and } \quad \lambda = 0\,,$$
  so that \eqref{mainineq} reduces to
\begin{multline}\label{mainineqwl}
c_{15} \left(\int_{\Omega_k} v_k^{(p+\beta-1)\frac{N}{N-2}}\,dx\right)^{\frac{N-2}{N}}   \le  \int_{\Omega_k} |f|^r\,dx+  \int_{\Omega_k} v_{k}^{(1-p+\beta)\frac{r}{r-2}}\,dx \\
  +
 \int_{\Omega_k}v_{k}^{\beta+2\gamma - p + 1}\,dx
 + \int_{\Omega_k} v_k^{p+\beta-1} \, dx\ .
\end{multline}
First, note that by H\"older inequality we get
    \begin{equation} \label{Hol}
    \int_{\Omega_k}v_{k}^{p+\beta-1} \leq |\Omega_k|^{\frac2N} \left(\int_{\Omega_k}v_{k}^{(p+\beta-1)\frac{N}{N-2}}\,dx \right)^{\frac{N-2}{N}}\,, 
    \end{equation}
and notice that,  $|\Omega_k| \le \frac{1}{k} \|\sqrt{|Du_\eps|^2 + 1}\|_{L^1(\Omega)}$. Thus choosing $k$ sufficiently large in order to have $|\Omega_k|^{\frac2N} \leq \frac{c_6}{2}$, we can absorb the last integral in \eqref{mainineqwl} in the left hand side.

Now we choose the parameters $r$ and $\beta$ in the following way 
    \[
   r=\frac{2}{N}  \frac{N \big(\gamma-(p-1)\big)}{\gamma}+\frac{N-2}{N}q \quad \mbox{ and } \quad 
     \beta=(r-2)\gamma+p-1\,.
    \]
    It is immediate to verify that
    \[
 \frac{r}{r-2}(\beta-p+1)=r\gamma=\beta+\eta
  \qquad \mbox{   and } \qquad 
   (\beta+p-1)\frac{N}{N-2}=q \gamma.
   \]
Finally we have, since $\frac{N\big(\gamma-(p-1)\big)}{\gamma}<r<q$, 
   \begin{equation}\label{fyoung}
 \int_{\Omega_k}|f|^r\,dx\leq  \|f\|_{L^{q}(\Omega)}^{r}|\Omega_k|^{\frac{q-r}{q}},
   \end{equation}
   and so plugging \eqref{Hol} and  \eqref{fyoung} into \eqref{mainineqwl}, by the choices of the parameters $r$ and $\beta$ and since  $|\Omega_k| \to 0$,   we end up with
   \[
c_{17} \left(\int_{\Omega_k} v_k^{r\gamma}\right)^{\frac{N-2}{N}}\leq \int_{\Omega_k} v_k^{r\gamma}+\omega(|\Omega_k|), \qquad \mbox{ where } 
\omega(|\Omega_k|) \to 0 \quad \mbox{ as } \quad |\Omega_k| \to 0\,.
   \]
   Hence  \eqref{crucialest} holds.
   
   \medskip
   
 $\bullet$  If $\lambda > 0$, then one concludes in the very same way replacing $f$ by $f - \lambda u$.
\medskip

   $\bullet$
   Suppose now  $${q}=\frac{N\big(\gamma-(p-1)\big)}{\gamma}  \qquad \text{ and } \qquad \lambda = 0 \,, $$ 
and choose now $r=q$. We can repeat the same proof of the previous case, with the only difference that we require  $\|f\|_q$ to be small, in order to deduce  \eqref{crucialest}.

\medskip

   $\bullet$ When $\lambda>0$, we can further exploit the zero-th other terms to avoid the smallness assumption on the datum $f$. Recalling  \eqref{ci1}   we have that 
   \begin{multline}\label{ci2}
c_{18} \left(\int_{\Omega_k} v_k^{r\gamma }\,dx\right)^{\frac{N-2}{N}} + \lambda c_{18} \int_{\Omega_k}|Du|^2\frac{ v_{k}^{\beta}}{v}\,dx  
\le  \int_{\Omega_k} |f|^2 v_{k}^{1-p+\beta} \,dx+  \\
+ \lambda^2 \int_{\Omega_k}|u|^2v_k^\beta v^{1-p}\,dx  +  \int_{\Omega_k} v_k^{p+\beta-1}+
2 \int_{\Omega_k}v_{k}^{q \gamma}\,dx.
\end{multline}
On the left-hand side, since $v \ge v_k$,
 \[
\lambda c_{18} \int_{\Omega_k}|Du|^2\frac{ v_{k}^{\beta}}{v}\,dx = \lambda c_{18} \int_{\Omega_k}  v v_{k}^{\beta} dx - \lambda c_{18} \eps \int_{\Omega_k} \frac{ v_{k}^{\beta}}{v}\,dx \ge \\
\lambda c_{18} k\int_{\Omega_k}v_{k}^{\beta}\,dx - \lambda c_{18}  \eps \int_{\Omega_k} v_{k}^{\beta-1}\,dx.
 \]
 Note that the last term can be absorbed into the leftmost term of \eqref{ci2} through the Young's inequality. On the right-hand side of \eqref{ci2}, first, for any $\omega > 0$ (recall that $\gamma q = (\beta + 1-p)\frac{q}{q-2}$)
\[
\lambda^2 \int_{\Omega_k}|u|^2 v_k^\beta v^{1-p}\,dx \le \lambda^2 \int_{\Omega_k}|u|^2v_k^{\beta + 1-p}\,dx \le c_{19}(\lambda)  \int_{\Omega_k} |u|^{q}\,dx
 + c_{20}\int_{\Omega_k} v_{k}^{\gamma q}\,dx\,.
\]
 Then, we need to show that $ \int_{\Omega_k}  |f|^2    v_{k}^{ \beta-(p-1)} $ can be made small for $k$ large. We then write
\begin{align*}
 \int_{\Omega_k}  |f|^2    v_{k}^{ \beta-(p-1)}  dx&= \int_{\Omega_k\cap \left\{ |f|^2 \leq k^{\frac{\beta-(p-1)}{\beta}}\right\}} 
  |f|^2    v_{k}^{ \beta- (p-1)}  dx
+
\int_{\Omega_k\cap \left\{ |f|^2 \geq k^{\frac{\beta-(p-1)}{\beta}}\right\}} 
|f|^2    v_{k}^{ \beta- (p-1)}  dx
\\
&\leq 
\int_{\Omega_k }  k^{\frac{\beta-(p-1)}{\beta}} 
    v_{k}^{ \beta- (p-1)}  dx
+
\int_{\Omega_k\cap \left\{ |f|^2 \geq k^{\frac{\beta-(p-1)}{\beta}}\right\}} 
|f|^2    v_{k}^{ \beta- (p-1)}  dx
\\
&\leq 
\lambda c_{18} k  \int_{\Omega_k } 
    v_{k}^{ \beta }  dx
 +c_{21}  |\Omega_k|
 +
 \int_{\Omega_k\cap \left\{ |f|^2 \geq k^{\frac{\beta-(p-1)}{\beta}}\right\}} 
|f|^q     dx  + 
 \int_{\Omega_k} 
    v_{k}^{ \frac{q}{q-2}(\beta- (p-1))}  dx  
\\
&= 
\lambda c_{18} k \int_{\Omega_k } 
    v_{k}^{ \beta }  dx
 +\int_{\Omega_k } 
    v_{k}^{  \gamma q}  dx   + \omega'(|\Omega_k|),
  \end{align*}
  where $\omega'(|\Omega_k|)$ as $|\Omega_k| \to 0$ (note that $\omega'$ depends on $f$). The conclusion now follows as before, but the estimate depends in addition on $\lambda, \|u\|_{L^q(\Omega)}$.

 \end{proof}

\section{Final remarks}
We list here some final comments and open problems.

\begin{rem}\label{regdom}
We do not pursue here the direction of weakening the regularity of the domain. We just point out that it can be considerably weakened, as soon as the Sobolev inequality and the divergence theorem hold. These facts were thoroughly discussed for the Dirichlet and Neumann problem for second order elliptic equations in Chapter 3 of \cite{Grisvard}, and recently considered in \cite{CianchiMazyaCpaa,CianchiMazyaJMPA} for the $p$-Poisson equation.
\end{rem}

\begin{rem}\label{so}
An inspection throughout the proofs suggest that our integral Bernstein argument provides some new second order integral estimates. These have been discussed in \cite{GL} for sublinear problems, and lead to completely new properties for stationary quasilinear elliptic equations with superlinear gradient terms. In particular, the techniques of the present manuscript allow to retrieve a result of \cite{CianchiMazyaJEMS}, see Remark 5.8 in \cite{GL}.
\end{rem}
\begin{rem}
We underline that when $\lambda>0$ the estimates in Theorem \ref{main2} may depend in addition on $\|\lambda u\|_r$. This quantity can be a priori bounded in many different cases (though results appeared mostly for problems with Dirichlet boundary conditions): test function arguments can be used when $\gamma<p$, see \cite{GMP14}. Morever, $L^\infty$-estimates can be achieved for problems having critical gradient growth $\gamma=p$ (and $r > N/p$), see \cite{ADP,PorrSegura}, whilst better estimates at the level of H\"older spaces can be obtained in the supernatural growth case $\gamma>p$, see \cite{DP,Glio}. Finally, one can add a (linear) first-order term $b\cdot Du$ with $b\in L^s(\Omega)$ for $s$ suitably large in Theorems \ref{main1} and \ref{main2} without substantial modifications in the proofs, as done in e.g. \cite{BardiPerthame,GP}. \end{rem}
\begin{rem}
It is still unclear whether the result in Theorem \ref{main2} continues to hold in the subquadratic case $p\in(1,2)$, and this remains at this stage an open problem.
\end{rem}

\begin{rem}
While the estimate in Theorem \ref{main1} leads to an explicit bound with respect to the $L^q$-norm of $f$, which agrees with the ones for the $p$-Poisson equation, the dependence with respect to $f$ in Theorem \ref{main2} is much more implicit. It is still unclear if the maximal regularity bound could be proved in the stronger form of Theorem \ref{main1}, even in the linear case $\tilde a(t) \equiv \tilde a > 0$.
\end{rem}

\begin{rem}
The techniques of the present manuscript do not apply to general parabolic quasilinear problems (unless 
$\partial_t f$ is proven to be regular enough). However, some advances for the evolutive counterpart of \eqref{divappHJ} with sublinear growth can be found in the recent paper \cite{GL}. 
\end{rem}
 
\begin{rem} 
Let us observe that, for $f_\eps$ smooth enough,  any solution $u_\eps$ to 
\begin{equation}\label{divappHJ2rem}
\begin{cases}
\lambda u_\eps -\mathrm{div}((\eps+|Du_\eps|^2)^{\frac{p-2}{2}}   Du_\eps)+|Du_\eps|^\gamma=f_\eps (x)\qquad &\text{ in }\Omega\ ,\\
\partial_\nu u_\eps=0&\text{ on } \partial\Omega\,,
\end{cases}
\end{equation}
is smooth, for $\lambda > 0 $,  $p>1$ and $\gamma>p-1$.  Indeed by \cite[Proposition 7.1]{LPcpde} we deduce that $u_\eps \in W^{1,\infty} (\Omega)$. This implies in particular that the lower order term is bounded, so that the standard regularity results (see \cite{Lieb91}
) imply that $u_\eps\in C^{1,\alpha}$. 
Arguing by bootstrap we get that $ -\mathrm{div}((\eps+|Du_\eps|^2)^{\frac{p-2}{2}}   Du_\eps)$ is H\"older continuous, and so by Schauder estimates  for uniformly elliptic problems $u_\eps$ is $C^3$ up to the boundary of $\Omega$.
We stress that the first (delicate) step, i.e. the Lipschitz estimate, in principle holds only for the $p$-Laplacian, with $p>2$. Anyway it is not hard to extend the results contained in  \cite{LPcpde} to a larger class of operators of the form $- \text{ div } \big( a(|Du|) D u\big) $, where $a$ satisfies \eqref{Aa}--\eqref{A1} and a technical assumption that involves the behavior of $a''$  with respect to $a'$ (the presence of this additional assumption  is not surprising since the method used in  \cite{LPcpde}  is based on the derivation of the equation) that is trivially satisfied by the $p$-Laplacian. 
 \end{rem}

 \begin{rem}[Existence  of solutions]
The estimates proved in our main results yield to the proof of the  existence of (at least) a solution to 
\begin{equation}\label{probl}
\begin{cases}
\lambda u -\mathrm{div}(|Du |^{p-2}   Du )+|Du |^\gamma=f  (x)\qquad &\text{ in }\Omega\ ,\\
\partial_\nu u =0&\text{ on } \partial\Omega\,. 
\end{cases}
\end{equation}
Assume indeed that   $\lambda > 0 $,  $\gamma>p-1$,    $f\in L^q (\Omega)$,
and either 
 $p>2$,   and  $q \geq \max\big\{\frac{N(\gamma-(p-1))}{\gamma},2 \big\}$ or 
 $p>1$,   and  $q  >N$. Then there exists at least a weak solution to \eqref{probl} that belongs to $W^{1,\gamma q} (\Omega)$. Indeed arguing by approximation, we consider a sequence of  (smooth) solutions $u_\eps$ to \eqref{divappHJ2rem}. Thus applying Theorem \ref{main1} combined with Remark \ref{remu} (or Theorem \ref{main2}), we get the  bound of $D u_\eps$ in $L^{q \gamma} (\Omega)$, and we also easily deduce that 
$$
\lambda  \|   u_\eps   \|_{L^1 (\Omega)} \leq \|f_\eps\|_{L^1 (\Omega)} + \| |D u_\eps|^\gamma \|_{L^1 (\Omega)}\,. 
$$

Since $- \Delta_p u_\eps$ is bounded in $L^1 (\Omega)$, then, up to a (not relabeled) subsequence, $D u_\eps$ a.e. converges to $D u$ by Lemma 1 in \cite{BGae}, and thanks to the bound  given by Theorem \ref{main1} we deduce the compactness of $u_\eps $ in  $W^{1,r} (\Omega)$ for any $ 1\leq r<\gamma q$. This is enough to pass to the limit in the weak formulation of the approximating problem and get a solution  to \eqref{probl}.

\end{rem}

\end{document}